\documentclass{amsart}

\usepackage{amsmath}
\usepackage{amsthm}
\usepackage{amsfonts}
\usepackage{dsfont}
\usepackage{color}

\newcommand{\reals}{\mathbb{R}}

\newcommand{\paraa}[1]{\big(#1\big)}
\newcommand{\parab}[1]{\Big(#1\Big)}
\newcommand{\parac}[1]{\bigg(#1\bigg)}

\newcommand{\sgn}{\operatorname{sgn}}

\newcommand{\spacearound}[1]{\quad#1\quad}
\newcommand{\equivalent}{\spacearound{\Leftrightarrow}}

\newtheorem{theorem}{Theorem}[section]

\newtheorem{lemma}[theorem]{Lemma}
\newtheorem{proposition}[theorem]{Proposition}
\newtheorem{example}[theorem]{Example}
\newtheorem{remark}[theorem]{Remark}
\theoremstyle{definition}

\numberwithin{equation}{section}

\newcommand{\Tr}{\operatorname{Tr}}

\newcommand{\B}{\mathcal{B}}
\renewcommand{\P}{\mathcal{P}}

\newcommand{\C}{\mathcal{C}}

\newcommand{\xv}{\vec{x}}

\renewcommand{\d}{\partial}
\newcommand{\TSigma}{T\Sigma}
\newcommand{\eps}{\varepsilon}

\newcommand{\nablab}{\bar{\nabla}}
\newcommand{\grad}{\operatorname{grad}}
\newcommand{\Gammab}{\bar{\Gamma}}
\newcommand{\Rb}{\bar{R}}

\newcommand{\av}{\vec{a}}
\newcommand{\cv}{\vec{c}}

\renewcommand{\mid}{\mathds{1}}
 
\newcommand{\D}{\mathcal{D}}

\newcommand{\nablah}{\hat{\nabla}}
\newcommand{\nablat}{\widetilde{\nabla}}

\renewcommand{\div}{\operatorname{div}}

\newcommand{\Ric}{\operatorname{Ric}}

\title[]{Pseudo-Riemannian geometry in\\ terms of multi-linear
  brackets}

\author{Joakim Arnlind}
\address[Joakim Arnlind]{Dept. of Math.\\
Link\"oping University\\
581 83 Link\"oping\\
Sweden}
\email{joakim.arnlind@liu.se}

\author{Gerhard Huisken}
\address[Gerhard Huisken]{Mathematisches Forschungsinstitut Oberwolfach\\
Schwarzwaldstr. 9-11\\
77709 Oberwolfach-Walke\\
Germany}
\email{huisken@mfo.de}

\thanks{}

\subjclass[2000]{}

\begin{document}

\begin{abstract}
  We show that the pseudo-Riemannian geometry of submanifolds can be
  formulated in terms of higher order multi-linear maps. In
  particular, we obtain a Poisson bracket formulation of almost
  (para-)K\"ahler geometry.
\end{abstract}

\maketitle

\section{Introduction}

\noindent In a series of papers, the possibility of expressing
differential geometry of Riemannian submanifolds as multi bracket algebraic
expressions in the function algebra has been investigated
\cite{ahh:discretegb,ahh:nsurface,ahh:psurface,ahh:nambudiscrete}.
More precisely, it was shown that on a $n$-dimensional submanifold
$\Sigma$, geometric objects can be written in terms of a $n$-ary
alternating multi-linear map acting on the embedding functions.  One
of the original motivations for studying the problem came from matrix
regularizations of surfaces in the context of ``Membrane Theory''
(cp. \cite{h:phdthesis}), where smooth functions are mapped to
hermitian matrices such that the Poisson bracket of functions
correspond to the commutator of matrices (as the matrix dimension
becomes large).  In this context, matrices corresponding to the
embedding coordinates of a surface arise as solutions to equations,
which contain matrices associated to surfaces of arbitrary genus. In
order to identify the topology of a solution, it is desirable to be
able to compute geometric invariants in terms of the embedding
matrices and their commutators. This was illustrated in
\cite{ahh:discretegb} where formulas for the discrete scalar curvature
and the discrete genus were presented (see also
\cite{a:curvatureGeometricModule}). Although matrix regularizations
provided the original motivation for our work, it is interesting to
ask similar questions in the context of general quantizations and
non-commutative geometry.

For higher dimensional manifolds, however, one is required to
formulate geometry in terms of a $n$-ary bracket, which has no direct
analogue as a higher order commutator for operators. This leads to the
question if there is perhaps a class of manifolds (of dimension
greater than two) for which one may use a Poisson bracket to express
geometric quantities. In the following we will demonstrate that
almost (para-)K\"ahler manifolds provide a context where an affirmative
answer can be given (cp. \cite{ah:kahlerpb} for a preliminary
version). In the course of doing so, we shall also consider
pseudo-Riemannian manifolds and extend the results
of \cite{ahh:nambudiscrete} to more general types of multi-linear
brackets and manifolds of indefinite signature.

\section{Preliminaries}\label{sec:preliminaries}

\noindent Let $(M,\eta)$ be a pseudo-Riemannian manifold of dimension
$m$, and let $(\Sigma,g)$ be a $n$-dimensional submanifold of $M$ with
induced metric $g$. Given local coordinates $x^1,\ldots,x^m$ on $M$,
we consider $\Sigma$ as embedded in $M$ via $x^i(u^1,\ldots,u^n)$
where $u^1,\ldots,u^n$ are local coordinates on $\Sigma$. Indices
$i,j,k,\ldots$ run from $1$ to $m$ and indices $a,b,c,\ldots$ run from
$1$ to $n$. The Levi-Civita connection on $M$ is denoted by $\nablab$
(with Christoffel symbols $\Gammab^i_{jk}$) and the Levi-Civita
connection on $\Sigma$ by $\nabla$ (with Christoffel symbols
$\Gamma^a_{bc}$).  The tangent space $\TSigma$ is regarded as a
subspace of the tangent space $TM$ and at each point of $\Sigma$ one
can choose $e_a=(\d_ax^i)\d_i$ as basis vectors of $\TSigma$, and in
this basis we define $g_{ab}=\eta(e_a,e_b)$.

The formulas of Gauss and Weingarten split the covariant derivative in
$M$ into tangential and normal components as
\begin{align}
  &\nablab_X Y = \nabla_X Y + \alpha(X,Y)\label{eq:GaussFormula}\\
  &\nablab_XN = -W_N(X) + D_XN\label{eq:WeingartenFormula}
\end{align}
where $X,Y\in \TSigma$, $N\in\TSigma^\perp$ and $\nabla_X Y$,
$W_N(X)\in\TSigma$ and $\alpha(X,Y)$, $D_XN\in\TSigma^\perp$. It
follows that $\alpha(X,Y)=\alpha(Y,X)$ and 
\begin{align*}
  \eta\paraa{\alpha(X,Y),N} = \eta\paraa{W_N(X),Y}
\end{align*}
for $N\in\TSigma^\perp$ (Weingarten's equation). From these formulas,
one can derive Gauss' equation, which expresses the curvature of the
submanifold in terms of the curvature of the ambient manifold and the
second fundamental form $\alpha$:
\begin{equation}\label{eq:GaussEquation}
  \begin{split}
    g\paraa{R(X,Y)Z,V} =
    \eta&\paraa{\Rb(X,Y)Z,V}-\eta\paraa{\alpha(X,Z),\alpha(Y,V)}\\
    &+\eta\paraa{\alpha(Y,Z),\alpha(X,V)},    
  \end{split}
\end{equation}
for $X,Y,Z,V\in\TSigma$, where $\Rb$ and $R$ denote the curvature
tensors of $M$ and $\Sigma$ respectively.  For more details on
submanifolds, please see e.g.
\cite{kn:foundationsDiffGeometryI,kn:foundationsDiffGeometryII}.

\section{$(N+1)$-bracket formulation of pseudo-Riemannian geometry}

\noindent In the previous section, we introduced $\Sigma$ as a
submanifold of $(M,\eta)$, embedded via the coordinates
$x^1,\ldots,x^m$, and equipped with the induced metric $g$. Let us now
assume that there exists a $(N+1)$-multilinear map
\begin{align*}
  \{\cdot,\ldots,\cdot\}: \underbrace{C^{\infty}(\Sigma)\times\cdots\times
  C^{\infty}(\Sigma)}_{N+1}\to C^{\infty}(\Sigma)
\end{align*}
compatible with the usual associative product in the following
way\footnote{Note that we have not assumed antisymmetry of the
  bracket; although our examples are of this kind, it is not necessary
  to develop the theory.}
\begin{align*}
  \{f_1,\ldots,f_kg,\ldots,f_{N+1}\}
  =f_k\{f_1,\ldots,g,\ldots,f_{N+1}\}
  +g\{f_1,\ldots,f_{N+1}\},
\end{align*}
for $k=1,\ldots,N+1$. Furthermore, we introduce multi-indices
$I=(i_1i_2\cdots i_N)$, $\av=(a_1\cdots a_N)$ and set
\begin{align*}
  &\{f,\xv^I\} = \{f,x^{i_1},x^{i_2},\ldots,x^{i_{N}}\}\\
  &\d_{\av}\xv^I = \paraa{\d_{a_1}x^{i_1}}\paraa{\d_{a_2}x^{i_2}}\cdots\paraa{\d_{a_{N}}x^{i_{N}}}\\
  &\eta_{IJ} = \eta_{i_1j_1}\eta_{i_2j_2}\cdots\eta_{i_{N}j_{N}}\\
  &g_{\av\cv}= g_{a_1c_1}g_{a_2c_2}\cdots g_{a_{N}c_{N}},
\end{align*}
as well as
\begin{align*}
  \P^{iI} = \frac{1}{\sqrt{N!}}\{x^i,\xv^I\}.
\end{align*}
Since we seek to formulate the metric geometry of $(\Sigma,g)$ with
the help of the above bracket, one needs to assume a relation to the
metric $g$. Hence, we will in the following assume that there exist
$0<\gamma\in C^{\infty}(\Sigma)$ and $\epsilon\in\{-1,1\}$ such that
\begin{align}\label{eq:PgDef}
  \P^{iI}\P^{jJ}\eta_{IJ} = \epsilon\gamma^2g^{ab}(\d_ax^i)(\d_bx^j),
\end{align}
where $g^{ab}$ are the components of the inverse of the metric
$g$. Introducing a multi-vector $\theta$ such that 
\begin{align*}
  \{f,f_1,\ldots,f_N\} = \theta^{a\av}(\d_af)(\d_{a_1}f_1)\cdots(\d_{a_N}f_N)
\end{align*}
equation \eqref{eq:PgDef} can be written as
\begin{align}\label{eq:gthetathetaRelation}
  \epsilon\gamma^2g^{ac} = \frac{1}{N!}\theta^{a\av}\theta^{c\cv}g_{\av\cv}.
\end{align}
(note that this kind of compatibility condition, and the relation to
Riemannian geometry has also been studied in the context of matrix
models \cite{bs:curvatureMatrix}). If desired, relation
\eqref{eq:PgDef} can be put in a slightly more algebraic form as
\begin{align*}
  \P^{iI}\eta_{IJ}\P^{jJ}\eta_{jk}\P^{k}(f_1,\ldots,f_N)
  = \epsilon\gamma^2\P^i(f_1,\ldots,f_N)
\end{align*}
for all $f_1,\ldots,f_N\in C^\infty(\Sigma)$, where
\begin{align*}
  \P^i(f_1,\ldots,f_N) = \frac{1}{\sqrt{N!}}\{x^i,f_1,\ldots,f_N\}.
\end{align*}
Although relation \eqref{eq:PgDef} might look unnatural at first
sight, let us point out a number of situations in which it holds true.

\begin{example}[Pseudo-Riemannian manifolds]\label{ex:pseudoRiemannian}
  Let $(\Sigma,g)$ be a pseudo-Riemannian manifold of dimension $n$ and set
  \begin{align*}
    \{f_1,\ldots,f_n\}=\frac{1}{\rho}\eps^{a_1\cdots a_n}
    (\d_{a_1}f_1)\cdots(\d_{a_n}f_n),
  \end{align*}
  (giving $N+1=n$) where $\rho$ is an arbitrary density. Then one computes
  \begin{align*}
    \P^{iI}\P^{jJ}\eta_{IJ} &= 
    \frac{1}{\rho^2(n-1)!}\eps^{aa_1\cdots a_{n-1}}\eps^{cc_1\cdots c_{n-1}}
    (\d_{a}x^i)(\d_cx^j)(\d_{a_1}x^{i_1})\eta_{i_1j_1}(\d_{c_1}x^{j_1})\\
    &\qquad\times\cdots\times(\d_{a_{n-1}}x^{i_{n-1}})\eta_{i_{n-1}j_{n-1}}(\d_{c_{n-1}}x^{j_{n-1}})\\
    &=\frac{1}{\rho^2(n-1)!}\eps^{aa_1\cdots a_{n-1}}\eps^{cc_1\cdots c_{n-1}}
    (\d_{a}x^i)(\d_cx^j)g_{a_1c_1}\cdots g_{a_{n-1}c_{n-1}}\\
    &=\frac{g}{\rho^2}g^{ac}(\d_{a}x^i)(\d_cx^j),
  \end{align*}
  since the next to last expression is simply the cofactor expansion
  of the matrix corresponding to the inverse of the metric $g$. Thus,
  \eqref{eq:PgDef} is fulfilled with $\epsilon=\sgn(g)$ and
  $\gamma=\sqrt{|g|}/\rho$, where $g$ denotes the determinant of the
  metric. In particular, one may use a Poisson bracket to describe the
  geometry of a 2-dimensional manifold of arbitrary signature. Note
  that the geometry of pseudo-Riemannian surfaces in terms of Poisson
  brackets was worked out in \cite{h:poissonpseudoRiemannian},
  following the work previously done in \cite{ahh:psurface}.
\end{example}

\begin{example}[Almost K\"ahler manifolds]\label{ex:kahler}
  Let $(\Sigma,g,J)$ be an almost K\"ahler manifold with the
  associated K\"ahler form
  \begin{align*}
    \omega(X,Y) = g(X,J(Y)),
  \end{align*}
  and Poisson bivector $\theta$ as the inverse of $\omega$ (where $J$
  denotes the almost complex structure).
  That is, in this setting one has $N=1$, and 
  \begin{align*}
    \{f_1,f_2\} = \theta^{ab}(\d_af_1)(\d_bf_2).
  \end{align*}
  On an almost K\"ahler manifold it holds that ${J^a}_b=-\theta^{ac}g_{cb}$,
  and $J^2=-\mid$ gives
  \begin{align}\label{eq:kahlergthetatheta}
    g^{ab} = \theta^{ap}\theta^{bq}g_{pq},
  \end{align}
  which implies that equation \eqref{eq:gthetathetaRelation} is
  satisfied with $\epsilon=\gamma=1$.
\end{example}

\begin{example}[Indefinite K\"ahler manifolds]\label{ex:indefiniteKahler}
  An indefinite (almost) K\"ahler manifold is an (almost) K\"ahler
  manifold where the metric is not necessarily positive
  definite. Since the complex structure preserves the causal type of
  vectors (spacelike, timelike or null), any subspace of vectors of a
  fixed causality is left invariant by the complex
  structure. Therefore, the index of $g$ (i.e. the dimension of the
  largest subspace on which $g$ is negative definite) has to be an
  even integer; that is, the signature of $g$ is of the form
  $(2s,2n-2s)$ for $0\leq s\leq n$.  In this case, equation
  \eqref{eq:kahlergthetatheta} still holds, which implies that
  \eqref{eq:gthetathetaRelation} is satisfied with
  $\epsilon=\gamma=1$, as in the case of ordinary K\"ahler
  manifolds. (Please see \cite{br:indefiniteKahler} for more
  information on indefinite K\"ahler manifolds.)
\end{example}

\begin{example}[Para-K\"ahler manifolds]\label{ex:paraKahler}
  An almost para-hermitian manifold $(\Sigma,g,J)$ is a
  pseudo-Riemannian manifold $(\Sigma,g)$ together with a map
  $J:\TSigma\to\TSigma$ such that $J^2=\mid$ and
  $g(X,Y)=-g(J(X),J(Y))$. If the associated K\"ahler form
  $\omega(X,Y)=g(X,JY)$ is closed, then $(\Sigma,g,J)$ is called an
  almost para-K\"ahler manifold. Since $J$ is invertible and maps a
  vector of negative norm to a vector of positive norm, the signature
  of $g$ has to be of the form $(n,n)$ (see
  e.g. \cite{cfg:surveyParacomplex} for more details). In this case,
  one derives that the Poisson bivector $\theta$ is given by
  $\theta^{ab}={J^a}_cg^{cb}$, and $J^2=\mid$ implies that
  \begin{align*}
    g^{ab} = -\theta^{ap}\theta^{bq}g_{pq}. 
  \end{align*}
  Thus, \eqref{eq:gthetathetaRelation} is fulfilled with $\epsilon=-1$
  and $\gamma=1$.
\end{example}

\begin{example}
  In the above examples, we have considered (para-)K\"ahler manifolds,
  which implies that $\gamma=1$. Let us note that manifolds fulfilling
  \eqref{eq:gthetathetaRelation}, with $\gamma\neq 1$, are related to
  (para-)K\"ahler manifolds by rescaling the metric. Namely, if
  $(\Sigma,g)$ is a Poisson manifold such that
  \begin{align*}
    \epsilon\gamma^2g^{ab}=\theta^{ap}\theta^{bq}g_{pq},
  \end{align*}
  then $\Sigma$ is a (para-)K\"ahler manifold with respect to the
  rescaled metric $\tilde{g}=\gamma^{-1}g$ and the (para-)complex
  structure ${J^a}_b=\epsilon\gamma^{-1}\theta^{ac}g_{cb}$.
\end{example}

\begin{remark}
  Note that \eqref{eq:gthetathetaRelation} provides a natural (at
  least in this context) generalization of K\"ahler manifolds to
  $(N+1)$-brackets. Namely, just as (in the case $\gamma=1$)
  \eqref{eq:kahlergthetatheta} expresses the fact that
  $\theta^{ac}g_{cb}$ squares to $-\delta^a_b$, equation
  \eqref{eq:gthetathetaRelation} tells us that the ''square'' of the
  multivector $\theta$ fulfills
  \begin{align*}
    \theta^{a\av}\theta_{b\av} = \delta^a_b,
  \end{align*}
  giving a relation between $\theta$ and the metric $g$ on $\Sigma$.
\end{remark}

\noindent The right hand side of equation \eqref{eq:PgDef} is more or less the
projection operator from $TM$ to $\TSigma$. Therefore, one
introduces 
\begin{align*}
  \D^{ij} = \frac{\epsilon}{\gamma^2}\P^{iI}\P^{jJ}\eta_{IJ}
  =g^{ab}(\d_ax^i)(\d_bx^j),
\end{align*}
and sets $\D(X)=\D^{ij}\eta_{jk}X^k\d_i$ for $X=X^i\d_i\in TM$;
moreover, one notes that $\D$ is symmetric, i.e. $\D^{ij}=\D^{ji}$.
The factor $\epsilon\gamma^2$ is not independent of the
bracket, and can be computed from it via
\begin{align}\label{eq:PPtraceGamma}
  \frac{\epsilon}{n}\P^{iI}\P^{jJ}\eta_{IJ}\eta_{ij} 
  = \frac{\epsilon^2\gamma^2}{n}g^{ab}g_{ab} = \gamma^2.
\end{align}

\begin{proposition}
  The map $\D:TM\to TM$ is the orthogonal projection onto $\TSigma$.
\end{proposition}

\begin{proof}
  One easily sees that $\D$ is symmetric with respect to $\eta$
  (giving an \emph{orthogonal} projection); namely, one computes
  \begin{align*}
    \eta(X,\D(Y)) = \eta_{ij}X^i\D^{jk}\eta_{kl}Y^l = 
    \eta_{kl}\D^{kj}\eta_{ji}X^iY^l = \eta(\D(X),Y)
  \end{align*}
  since $\D^{jk}=\D^{kj}$. Moreover, using equation \eqref{eq:PgDef}
  one computes that
  \begin{align*}
    (\D^2)^{ij} &= \D^{ik}\eta_{kl}\D^{lj}
    =g^{ab}(\d_ax^i)(\d_bx^k)\eta_{kl}g^{pq}(\d_px^l)(\d_qx^j)\\
    &=g^{ab}g^{pq}g_{bp}(\d_ax^i)(\d_qx^j) = g^{aq}(\d_ax^i)(\d_qx^j)
    = \D^{ij},
  \end{align*}
  which shows that $\D$ is indeed projection operator. Now, let us choose
  $X\in\TSigma$ and write $X=X^a(\d_ax^i)\d_i$. One then computes
  \begin{align*}
    \D(X) &= g^{ab}(\d_ax^i)(\d_bx^j)\eta_{jk}X^c(\d_cx^k)\d_i
    =g^{ab}g_{bc}(\d_ax^i)X^c\d_i = X,
  \end{align*}
  showing that $\D$ is indeed the projection onto $\TSigma$.
\end{proof}

\noindent It is also convenient to introduce the projection $\Pi$ onto the
complementary space $\TSigma^\perp$; i.e we set $\Pi=\mid-\D$. Having
the projection at hand, one immediately obtains the
Levi-Civita connection of $\Sigma$ as
\begin{align*}
  \nabla_XY^i = \D\paraa{\nablab_XY}^i
  = {\D^i}_j\paraa{X^k\d_k(Y^j)+\Gammab^{j}_{kl}X^kY^l}
\end{align*}
for $X,Y\in\TSigma$. However, the above formula has an explicit
derivative appearing in it, and can not be completely written in terms
of $(N+1)$-brackets. Therefore, it is convenient to introduce
\begin{align*}
  \nablah_XY = \nablab_{\D(X)}Y,
\end{align*}
for which it holds that $\nablah_XY = \nablab_XY$, whenever
$X\in\TSigma$, and
\begin{align}\label{eq:nablahDef}
  \nablah_jY^i = \D_j(Y^i) + {\D_j}^l\Gammab^i_{lk}Y^k 
\end{align}
with
\begin{align*}
  \D^i(f) = \D^{ij}(\d_jf) =  
  \frac{\epsilon}{\gamma^2N!}\{x^i,\xv^I\}\{f,\xv^J\}\eta_{IJ}.
\end{align*}
Note that equation \eqref{eq:nablahDef} is written entirely in terms
of $(N+1)$-brackets:
\begin{align*}
  \nablah^jY^i =
  \frac{\epsilon}{\gamma^2N!}\{Y^i,\xv^I\}\{x^j,\xv^J\}\eta_{IJ}
  +\frac{\epsilon}{\gamma^2N!}\{x^j,\xv^I\}\{x^l,\xv^J\}\eta_{IJ}
  \Gammab^{i}_{lk}Y^k
\end{align*}
Thus, the Levi-Civita connection on $\Sigma$ may also be written in
terms of $(N+1)$-brackets as
\begin{align}\label{eq:LeviCivitaSigma}
  \nabla_jY^i = {\D^i}_k\nablah_jY^k.
\end{align}

\noindent Let us proceed to show that $\nablah$ is (not surprisingly)
closely related to the curvature of $(M,\eta)$. To start with, let us
collect a few computations related to the second fundamental form in
the following lemma:
\begin{lemma}\label{lemma:secondFundamentalForm}
  For $X,Y,Z,V\in\TSigma$ it holds that
  \begin{align}
    &\alpha(X,Y)^i = -\paraa{\nablah_j{\Pi^i}_k}X^jY^k\label{eq:secfundExpression}\\
    &\eta\paraa{\alpha(X,Y),\alpha(Z,V)}=
    \paraa{\nablah_i\Pi_{mj}}\paraa{\nablah_k{\Pi^m}_l}X^iY^jZ^kV^l.
  \end{align}
  In particular, since $\alpha(X,Y)=\alpha(Y,X)$ it follows from
  \eqref{eq:secfundExpression} that 
  \begin{align}\label{eq:nablahDsymmetry}
    X^iY^j\paraa{\nablah_i\D_{jk}-\nablah_j\D_{ik}} = 0.
  \end{align}
\end{lemma}
\begin{proof}
  For
  $X,Y\in\TSigma$, the second fundamental form is given by
  \begin{align*}
    \alpha(X,Y)^i = \Pi\paraa{\nablab_XY}^i = \Pi\paraa{\nablah_XY}^i
    ={\Pi^i}_kX^l\nablah_lY^k
    =-X^lY^k\nablah_l{\Pi^i}_k
  \end{align*}
  since $\Pi(Y)=0$. The second formula follows immediately from this
  result. The last formula is proved by inserting $\Pi=\mid-\D$ into
  equation \eqref{eq:secfundExpression} and using that
  $\alpha(X,Y)=\alpha(Y,X)$.
\end{proof}

\noindent The next results confirms that the commutator of $\nablah_i$
and $\nablah_i$ does indeed give the curvature of $(M,\eta)$:

\begin{proposition}
  Let $\Rb$ be the curvature tensor of $(M,\eta)$. For $X,Y\in\TSigma$
  and $U\in TM$ it holds that
  \begin{align*}
    X^iY^j\paraa{\nablah_i\nablah_jU^k-\nablah_j\nablah_iU^k} = 
    \Rb(X,Y)U^k
  \end{align*}
\end{proposition}

\begin{proof}
  From the definition of $\nablah$ one obtains
  \begin{align*}
    X^iY^j&\paraa{\nablah_i\nablah_jU^k-\nablah_j\nablah_iU^k}
    = X^iY^j\paraa{{\D_i}^l\nablab_l({\D_j}^m\nablab_mU^k)
      -{\D_j}^l\nablab_l({\D_i}^m\nablab_mU^k)}\\
    &= X^iY^j\paraa{[\nablab_i,\nablab_j]U^k
      +\nablah_i({\D_j}^m)\nablab_mU^k-\nablah_j({\D_i}^m)\nablab_mU^k}\\
    &= X^iY^j[\nablab_i,\nablab_j]U^k = \Rb(X,Y)U^k,
  \end{align*}
  by using equation \eqref{eq:nablahDsymmetry} in Lemma
  \ref{lemma:secondFundamentalForm}.
\end{proof}

\noindent By using Gauss' equation \eqref{eq:GaussEquation} and Lemma
\ref{lemma:secondFundamentalForm}, we proceed to show that the
curvature of $(\Sigma,g)$ can be expressed in terms of $\nablah$ and
the projection $\Pi$.

\begin{proposition}
  Let $\Rb$ and $R$ be the curvature tensors of $(M,\eta)$ and
  $(\Sigma,g)$ respectively. For $X,Y,Z,V\in\TSigma$ it holds that
  \begin{align*}
    R(X,Y,Z,V) = \parab{\Rb_{ijkl}+(\nablah_k{\Pi_{mi}})(\nablah_l{\Pi^m}_j)-
    (\nablah_k{\Pi_{mj}})(\nablah_l{\Pi^m}_i)}X^iY^jZ^kV^l.
  \end{align*}
\end{proposition}

\begin{proof}
  The formula for the curvature of $(\Sigma,g)$ is obtained by
  inserting the expression for the second fundamental form, found in
  equation \eqref{eq:secfundExpression} in Lemma
  \ref{lemma:secondFundamentalForm}, into Gauss' equation
  \eqref{eq:GaussEquation}.
\end{proof}

\noindent To derive formulas for the scalar and Ricci curvatures one
notes that the trace of $T:\TSigma\times \TSigma\to C^\infty(\Sigma)$
may be computed as
\begin{align*}
  \Tr T = g^{ab}T(e_a,e_b) = g^{ab}T_{ij}(e_a)^i(e_b)^j
  = g^{ab}(\d_ax^i)(\d_bx^j)T_{ij}=\D^{ij}T_{ij}.
\end{align*}
Thus, when computing the trace over $\TSigma$, one may effectively use
$\D^{ij}$ instead of $\eta^{ij}$ (which corresponds to the trace in
$TM$). In this way, one immediately obtains formulas for the scalar and Ricci
curvatures:
\begin{proposition}
  Let $\Rb$ denote the curvature tensor of $(M,\eta)$, and let $\Ric$
  and $S$ denote the Ricci and scalar curvatures of $(\Sigma,g)$,
  respectively. For $X,Y\in\TSigma$ it holds that
  \begin{align*}
    \Ric(X,Y) &= \parab{\D^{kl}\Rb_{kilj}
      +(\nablah_k\Pi^{kl})(\nablah_j\Pi_{il})
      -(\nablah_k\Pi_{li})(\nablah_j\Pi^{lk})
    }X^iY^j\\
    S &= \D^{ij}\D^{kl}\Rb_{kilj}
    +(\nablah^k\Pi_{kl})(\nablah_i\Pi^{il})
    -(\nablah_k\Pi_{il})(\nablah^i\Pi^{kl}).
  \end{align*}
\end{proposition}

\noindent Gauss' equation relates the curvature in the tangential direction
(i.e. along the submanifold) to the curvature of the ambient
space. Let us now study curvature in the normal directions; for this
reason, we introduce
\begin{align*}
  \B_N^{ij} = -\nablah^iN^j
  =-\frac{\epsilon}{\gamma^2N!}\{x^i,\xv^I\}\{N^j,\xv^J\}\eta_{IJ}
  -\frac{\epsilon}{\gamma^2N!}\{x^i,\xv^I\}\{x^k,\xv^J\}\eta_{IJ}\Gammab^{j}_{kl}N^l,
\end{align*}
for any $N=N^i\d_i\in\TSigma^\perp$, and note that this is again an
expression in terms of $(N+1)$-brackets. From this definition we
induce a mapping $\B_N:TM\to TM$ by setting
$\B_N(X)=\B_N^{ij}\eta_{jk}X^k\d_i$. It is clear from the definition
that $\B_N(X)\in\TSigma$ for any $X\in TM$ and
$N\in\TSigma^\perp$. Furthermore, it turns out that $\B_N$ captures
the tangential and normal components of $\nablab_XN$ in the following
way:

\begin{proposition}\label{prop:BWeingarten}
  For $X\in\TSigma$ and $N,\tilde{N}\in\TSigma^\perp$ it holds that 
  \begin{align}
    &\B_N(X) = W_N(X)\label{eq:BWrelation}\\
    &\eta\paraa{\B_N(\tilde{N}),X}=-\eta\paraa{D_XN,\tilde{N}},\label{eq:BDRelation}
  \end{align}
  where $W_N(X)$ and $D_XN$ denote the tangential and normal
  components of $\nablab_XN$ respectively
  (cp. equation \eqref{eq:WeingartenFormula}).
\end{proposition}

\begin{proof}
  Let us first note that, due to the symmetry of $\alpha$, it holds
  that
  \begin{align*}
    &\eta\paraa{\alpha(X,Y),N} = \eta\paraa{\alpha(Y,X),N}\equivalent
    \eta\paraa{W_N(X),Y} = \eta\paraa{W_N(Y),X}\\
    &\equivalent\eta\paraa{\nablab_XN,Y}=\eta\paraa{\nablab_YN,X},
  \end{align*}
  for $X,Y\in\TSigma$ and $N\in\TSigma^\perp$. Using this, one
  computes that
  \begin{align*}
    \B_N(X)^i &= -\D^{ik}(\nablab_kN^j)X_j
    =-g^{ab}(\d_ax^i)(\d_bx^k)\paraa{\nablab_kN^j}(\d_cx^l)\eta_{jl}X^c\\
    &=-g^{ab}(\d_ax^i)\eta\paraa{\nablab_{e_b}N,e_c}X^c
    =-g^{ab}(\d_ax^i)\eta\paraa{\nablab_{e_c}N,e_b}X^c\\
    &= -g^{ab}(\d_ax^i)\eta\paraa{\nablab_{X}N,e_b}
    =-\D\paraa{\nablab_XN}^i = W_N(X)^i.
  \end{align*}
  Moreover, for $N,\tilde{N}\in\TSigma^\perp$ and $X\in\TSigma$ one
  obtains
  \begin{align*}
    \eta\paraa{\B_N(\tilde{N}),X} &= -X_i\paraa{\nablah^iN^j}\tilde{N}_j
    =-\eta\paraa{\nablab_XN,\tilde{N}}=
    -\eta\paraa{D_XN,\tilde{N}},
  \end{align*}
  which proves formula \eqref{eq:BDRelation}.
\end{proof}

\noindent Note that since $\eta(W_N(X),Y)=\eta(\alpha(X,Y),N)$
(Weingarten's equation), it follows from Proposition
\ref{prop:BWeingarten} that
\begin{align}\label{eq:alphaBRelation}
  \eta\paraa{\alpha(X,Y),N} = \eta\paraa{\B_N(X),Y}.
\end{align}
As a complement to Gauss' equation, the Codazzi-Mainardi equations
express the normal component of the curvature in the ambient
space. Due to the symmetries of the curvature tensor, one can
immediately write
\begin{align*}
  \eta\paraa{\Rb(X,Y)Z,N} = -\eta\paraa{\Rb(X,Y)N,Z}
  =-X^iY^jZ^l\paraa{\nablah_j\nablah_iN_l-\nablah_i\nablah_jN_l},
\end{align*}
and from the definition of $\B_N^{ij}$ one obtains
\begin{align}\label{eq:BCodazzi}
  \begin{split}
    \eta\paraa{\Rb(X,Y)Z,N}
    &=X_iY_jZ_l\paraa{\nablah^j\B_N^{il}-\nabla^i\B_N^{jl}}\\
    &=\eta\paraa{(\nablah_Y\B_N)(Z),X}-\eta\paraa{(\nablah_X\B_N)(Z),Y}.
  \end{split}
\end{align}
Let us point out that this way of writing the normal component of the
curvature has a close resemblance to an expression in terms of a
connection in $\TSigma\oplus\TSigma^\perp$, defined by combining $D_X$
and $\nabla_X$. Namely, by writing
\begin{align*}
  \paraa{\nablat_X\alpha}(Y,Z) = D_X\alpha(Y,Z)-\alpha(\nabla_XY,Z)-\alpha(Y,\nabla_XZ) 
\end{align*}
the normal component of the curvature is 
$(\nablat_X\alpha)(Y,Z)-(\nablat_Y\alpha)(X,Z)$ (see
\cite{kn:foundationsDiffGeometryII}, page 25). Let us verify
directly that they are, in fact, the same. One computes
\begin{align*}
  \eta&\paraa{(\nablah_X\B_N)(Z),Y} 
  =\eta\paraa{(\nablab_X\B_N)(Z),Y}
  =\eta\paraa{\nablab_X\B_N(Z),Y}-\eta\paraa{\B_N(\nablab_XZ),Y}\\
  &=\eta\paraa{\nabla_X\B_N(Z),Y}-\eta\paraa{\B_N(\nabla_XZ),Y}-\eta\paraa{\B_N\alpha(X,Z),Y}\\
  &=X\cdot\eta\paraa{\B_N(Z),Y}-\eta\paraa{\B_N(Z),\nabla_XY}
  -\eta\paraa{\B_N(\nabla_XZ),Y}-\eta\paraa{\B_N\alpha(X,Z),Y},
\end{align*}
and using Proposition \ref{prop:BWeingarten} together with equation
\eqref{eq:alphaBRelation} one gets
\begin{align*}
  \eta&\paraa{(\nablah_X\B_N)(Z),Y} = 
  X\cdot\eta\paraa{\alpha(Y,Z),N}-\eta\paraa{\alpha(\nabla_XY,Z),N}\\
  &\qquad\qquad-\eta\paraa{\alpha(Y,\nabla_XZ),N}
  +\eta\paraa{D_YN,\alpha(X,Z)}\\
  &= X\cdot\eta\paraa{\alpha(Y,Z),N} + Y\cdot\eta\paraa{\alpha(X,Z),N}
  -\eta\paraa{D_Y\alpha(X,Z),N}\\
  &\quad-\eta\paraa{\alpha(\nabla_XY,Z),N}-\eta\paraa{\alpha(Y,\nabla_XZ),N},
\end{align*}
which implies that 
\begin{align*}
  \eta\paraa{(\nablah_Y\B_N)(Z),X}-\eta\paraa{(\nablah_X\B_N)(Z),Y}
  =\eta\paraa{(\nablat_X\alpha)(Y,Z)-(\nablat_Y\alpha)(X,Z),N}.
\end{align*}

\section{The Laplace operator}

\noindent The gradient of a function $f\in C^\infty(\Sigma)$ may be written as
\begin{align*}
  \D^i(f)\d_i= g^{ab}(\d_af)(\d_bx^i)\d_i = g^{ab}(\d_af)e_b 
  =\nabla f,
\end{align*}
and the divergence of an element $X\in\TSigma$ as
\begin{align*}
  \nablah_iX^i &= {\D^k}_i\nablab_kX^i = g^{ab}(\d_ax^k)(\d_bx^l)\eta_{li}\nablab_kX^i
  =g^{ab}(\d_bx^l)\eta_{li}(\nablab_{e_a}X)^i\\
  &=g^{ab}(\d_bx^l)\eta_{li}\paraa{\nabla_{e_a}X+\alpha(e_a,X)}^i
  =g^{ab}(\d_bx^l)\eta_{li}\paraa{\nabla_{e_a}X}^i\\
  &=g^{ab}(\d_bx^l)\eta_{li}\paraa{\nabla_{a}X}^c(\d_cx^i)
  =g^{ab}g_{bc}\nabla_aX^c = \nabla_aX^a = \div(X).
\end{align*}
Consequently, the Laplace-Beltrami operator on $(\Sigma,g)$ may be computed as
\begin{align*}
  \Delta(f) = \div\paraa{\grad(f)} = \nablah_i\nablah^i(f),
\end{align*}
where $\nablah^i(f) = \D^i(f)$. Let us now show that one may derive a
simpler form of the Laplace operator in certain special cases. Namely,
let us first assume that the multivector $\theta$ (defining the
$(N+1)$-bracket) is completely antisymmetric, which implies that
\begin{align*}  
  \{\gamma^M\{f,\xv^I\}\eta_{IJ},\xv^J\}
  &=\frac{1}{N!}\theta^{a\av}\d_a\parab{\gamma^M\theta^{c\cv}(\d_cf)(\d_{\cv}\xv^I)\eta_{IJ}}(\d_{\av}\xv^J)\\
  =\frac{1}{N!}\theta^{a\av}\d_a&\parab{\gamma^M\theta^{c\cv}(\d_cf)(\d_{\cv}\xv^I)(\d_{\av}\xv^J)\eta_{IJ}}
  =\frac{1}{N!}\theta^{a\av}\d_a\parab{\gamma^M\theta^{c\cv}(\d_cf)g_{\av\cv}},
\end{align*}
since $\theta^{a\av}\d_a(\d_{\av}\xv^J)=0$ due to the antisymmetry of
$\theta$. Next, assuming that there exists a function $\rho\in
C^\infty(\Sigma)$ such that $\d_a(\rho\theta^{a\av})=0$ one obtains
\begin{align*}
  \{\gamma^M\{f,\xv^I\}\eta_{IJ},\xv^J\}
  =\frac{1}{\rho N!}\d_a\paraa{\rho\gamma^M\theta^{a\av}\theta^{c\cv}g_{\av\cv}\,\d_cf}
  =\frac{\epsilon}{\rho}\d_a\paraa{\rho\gamma^{M+2}g^{ac}\d_cf},
\end{align*}
by using \eqref{eq:gthetathetaRelation}. In the case of Example
\ref{ex:pseudoRiemannian}, $\theta^{a\av}=\rho^{-1}\eps^{a\av}$, and it
follows immediately that $\theta$ is completely antisymmetric and that
$\d_a(\rho\theta^{a\av})=0$. Therefore, one gets
\begin{align*}
  \{\gamma^M\{f,\xv^I\}\eta_{IJ},\xv^J\}
  =\frac{\epsilon}{\rho}\d_a\paraa{\rho\gamma^{M+2}g^{ac}\d_cf}
  =\frac{\epsilon}{\rho}\d_a\paraa{\rho^{-M-1}\sqrt{|g|}^{M+2}g^{ac}\d_cf}
\end{align*}
since $\gamma=\sqrt{|g|}/\rho$, and choosing $M=-1$ gives
\begin{align}\label{eq:nambuLaplace}
  \Delta(f) = \frac{\eps}{\gamma}\{\gamma^{-1}\{f,\xv^I\}\eta_{IJ},\xv^J\}.
\end{align}
In Examples \ref{ex:kahler}--\ref{ex:paraKahler} the $(N+1)$-bracket
is a Poisson structure, and by setting $\rho=(\sqrt{\det\theta})^{-1}$
it follows from the Jacobi identity that $\d_a(\rho\theta^{ab})=0$;
namely, by multiplying the
Jacobi identity by $\omega$, the inverse of $\theta$
(i.e. $\omega_{ab}\theta^{bc}=\delta_a^c$), one obtains
\begin{align*}
  &\omega_{ab}\paraa{\theta^{ap}\d_p\theta^{bc}+\theta^{bp}\d_p\theta^{ca}+\theta^{cp}\d_p\theta^{ab}}
  = 0\equivalent\\
  &\d_a\theta^{ac}=\frac{1}{2}\omega_{ab}(\d_p\theta^{ba})\theta^{pc} 
  =\frac{1}{\det\theta}(\d_p\det\theta)\theta^{pc},
\end{align*}
since $\d_p\det\theta = (\det\theta)\omega_{ab}\d_p\theta^{ba}$, from
which it follows that
$\d_a\paraa{\sqrt{\det\theta}^{-1}\theta^{ac}}=0$ (cp. also
\cite{bs:curvatureMatrix}). These considerations imply that
\begin{align*}
  \{\gamma^M\{f,\xv^I\}\eta_{IJ},\xv^J\}
  =\epsilon\sqrt{\theta}\d_a\paraa{\sqrt{\det\theta}^{-1}\gamma^{M+2}g^{ac}\d_cf}.
\end{align*}
Now, the determinant of the relation $\eps\gamma^2g^{ab} =
\theta^{ap}\theta^{bq}g_{pq}$, together with $\dim\Sigma$ being even,
implies that
\begin{align*}
  \sqrt{\det\theta} = \frac{\sqrt{\gamma^{n}}}{\sqrt{|g|}}.
\end{align*}
Thus, by choosing $M=(n-4)/2$, the Laplace operator may
be written as
\begin{align}\label{eq:kahlerLaplace}
  \Delta(f) = \frac{\epsilon}{\sqrt{\gamma^{n}}}\{\sqrt{\gamma^{n-4}}\{f,\xv^I\}\eta_{IJ},\xv^J\},
\end{align}
for almost (para-)K\"ahler manifolds.

\section{Explicit formulas when $(M,\eta)=(\reals^m,\delta)$}

\noindent In the course of rewriting geometry in terms of
$(N+1)$-brackets, we have developed a notation which makes the
expressions for most quantities rather short and concise. Of course,
in the general case, writing out all the brackets and Christoffel
symbols will produce formulas that are quite lengthy. However, in the
particular situation when the ambient space is $\reals^m$ equipped
with the metric $g_{ab}=\delta_{ab}$ explicit expressions are
considerably reduced in size. Apart from being simple, it is also an
interesting case since it is generic, in the sense that any manifold
can be isometrically embedded in some Euclidean space
$(\reals^m,\delta)$ (\cite{n:imbeddingProblem}), and many manifolds do
have a concrete presentation in such a way. Let us therefore, in this
case, present explicit formulas for some of the geometric objects for
which we have developed an $(N+1)$-bracket formulation. Note that for
pseudo-Riemannian manifolds a similar statement holds where one may
always isometrically embed a pseudo-Riemannian manifold into
pseudo-Euclidean space (see \cite{c:pseudoImbeddingProblem} for
details). The formulas given below can easily be extended to this
setting.

When $(M,\eta)=(\reals^m,\delta)$ there is no difference between upper
and lower indices; therefore, we shall leave all (multi)-indices in
the upper position and assume that all repeated (multi-)indices are
summed over from $1$ to $m$. Moreover, since $(\reals^m,\delta)$ is
flat, it holds that $\nablah^iX^j = \D^i(X^j)$. The factor $\gamma^2$
will be eliminated via
\begin{align*}
  \gamma^2 = \frac{\epsilon}{n}\P^{iI}\P^{iI} = \frac{\epsilon}{nN!}
  \{x^i,\xv^I\}\{x^i,\xv^I\},
\end{align*}
giving, for instance,
\begin{align*}
  \D^{ik} &= \frac{\epsilon}{\gamma^2N!}\{x^i,\xv^I\}\{x^k,\xv^I\}
  =n\frac{\{x^i,\xv^I\}\{x^k,\xv^I\}}{\{x^j,\xv^J\}\{x^j,\xv^J\}}.
\end{align*}
In the same way, one derives the following expressions:
\begin{align*}
  \nabla^{i}(f)&=n\frac{\{f,\xv^I\}\{x^i,\xv^I\}}{\{x^j,\xv^J\}\{x^j,\xv^J\}}\qquad\qquad
  \div(X)=n\frac{\{X^i,\xv^I\}\{x^i,\xv^I\}}{\{x^j,\xv^J\}\{x^j,\xv^J\}}\\
  W_N(X)^i &= -n\frac{\{N^j,\xv^i\}\{x^i,\xv^I\}}{\{x^j,\xv^J\}\{x^j,\xv^J\}}X^j\qquad
  \Delta(f) = n^2\frac{\left\{
      \frac{\{f,\xv^L\}\{x^i,\xv^L\}}{\{x^k,\xv^K\}\{x^k,\xv^K\}}
      ,\xv^I\right\}\{x^i,\xv^I\}}{\{x^j,\xv^J\}\{x^j,\xv^J\}}\\
  \nabla_XY^i&=n^2\frac{X^k\{x^k,\xv^J\}\{Y^j,\xv^J\}\{x^j,\xv^I\}\{x^i,\xv^I\}}{\paraa{\{x^l,\xv^L\}\{x^l,\xv^L\}}^2}.
\end{align*}
The scalar curvature may be written as
\begin{align*}
  S=n^4&\frac{\{x^k,\xv^I\}\left\{\frac{\{x^k,\xv^K\}\{x^l,\xv^K\}}{\{x^j,\xv^{L}\}\{x^j,\xv^{L}\}},\xv^I\right\}\{x^i,\xv^J\}\left\{\frac{\{x^i,\xv^{A}\}\{x^l,\xv^{A}\}}{\{x^{j'},\xv^{B}\}\{x^{j'},\xv^{B}\}},\xv^J\right\}}{\paraa{\{x^m,\xv^M\}\{x^m,\xv^M\}}^2}\\
  &-n^4\frac{\{x^k,\xv^I\}\left\{\frac{\{x^i,\xv^K\}\{x^l,\xv^K\}}{\{x^j,\xv^{L}\}\{x^j,\xv^{L}\}},\xv^I\right\}\{x^i,\xv^J\}\left\{\frac{\{x^k,\xv^{A}\}\{x^l,\xv^{A}\}}{\{x^{j'},\xv^{B}\}\{x^{j'},\xv^{B}\}},\xv^J\right\}}{\paraa{\{x^m,\xv^M\}\{x^m,\xv^M\}}^2},
\end{align*}
and the Codazzi-Mainardi equations become
\begin{align*}
  \parac{
    \left\{\frac{\{N^k,\xv^K\}\{x^j,\xv^K\}}{\{x^l,\xv^L\}\{x^l,\xv^L\}},\xv^I\right\}\{x^i,\xv^I\}
    -\left\{\frac{\{N^k,\xv^K\}\{x^i,\xv^K\}}{\{x^l,\xv^L\}\{x^l,\xv^L\}},\xv^I\right\}\{x^j,\xv^I\}
  }X^iY^jZ^k
  =0,
\end{align*}
for all $X,Y,Z\in\TSigma$ and $N\in\TSigma^\perp$. For almost
(para-)K\"ahler manifolds, in which case $\gamma=1$ and $N=1$, the
formulas are even more compelling:
\begin{align*}
  \nabla^{i}(f)&=\epsilon\{f,x^j\}\{x^i,x^j\}\qquad\qquad
  \div(X)=\epsilon\{X^i,x^j\}\{x^i,x^j\}\\
  W_N(X)^i &= -\epsilon\{N^j,x^k\}\{x^i,x^k\}X^j\qquad
  \Delta(f) = \left\{\{f,x^k\}\{x^i,x^k\},x^j\right\}\{x^i,x^j\}\\
  \nabla_XY^i&=X^k\{x^k,x^l\}\{Y^j,x^l\}\{x^j,x^m\}\{x^i,x^m\},
\end{align*}
and the scalar curvature becomes
\begin{align*}
  S=&\{x^k,x^{i'}\}\left\{\{x^k,x^j\}\{x^l,x^j\},x^{i'}\right\}\{x^i,x^{l'}\}\left\{\{x^i,x^{k'}\}\{x^l,x^{k'}\},x^{l'}\right\}\\
  &-\{x^k,x^{i'}\}\left\{\{x^i,x^j\}\{x^l,x^j\},x^{i'}\right\}\{x^i,x^{l'}\}\left\{\{x^k,x^{k'}\}\{x^l,x^{k'}\},x^{l'}\right\}.
\end{align*}
Let us end this section by noting that, when $N=1$, the
Codazzi-Mainardi equations are actually Poisson algebraic identities
once we assume that equation \eqref{eq:PgDef} holds. Namely,
multiplying \eqref{eq:PgDef} by $\P^{jk}$ gives $\D^{ij}\P^{jk} =
\P^{ik}$ from which it follows that $\D^{ij}\P^j(f)=\P^i(f)$ and
$\P^{ij}\D^j(f)=\P^i(f)$ where $\P^i(f)=\{x^i,f\}$. Then one readily
proves the following:

\begin{lemma}\label{lemma:Dcommutation}
  If $\D^{ij}\P^{jk}=\P^{ik}$ then it holds that
  \begin{align*}
    [\D^i,\D^j](f)\P^{ik}\P^{jl} = 0 
  \end{align*}
  for $k,l=1,\ldots,m$ and $f\in\C^\infty(\Sigma)$. 
\end{lemma}

\begin{proof}
  One computes
  \begin{align*}
    [\D^i,\D^j]&(f)\P^{ik}\P^{jl} =
    \parab{\D^i(\D^j(f))-\D^j(\D^i(f))}\P^{ik}\P^{jl}\\
    &=-\P^k\paraa{\D^j(f)}\P^{jl}+\P^l\paraa{\D^i(f)}\P^{ik}\\
    &=-\P^k\paraa{\P^{jl}\D^j(f)}+\P^k\paraa{\P^{jl}}\D^j(f)
    +\P^l\paraa{\P^{ik}\D^i(f)}-\P^l\paraa{\P^{ik}}\D^i(f)\\
    &= \P^k\paraa{\P^l(f)}-\P^l\paraa{\P^{k}(f)}
    +\D^i(f)\parab{\P^k(\P^{il})-\P^l(\P^{ik})}
  \end{align*}
  by using $\D^{ij}\P^{jk}=\P^{ik}$ and the fact that $\P^i$ and $\D^i$
  are derivations. Using the Jacobi identity in the last term yields
  \begin{align*}
    [\D^i,\D^j](f)\P^{ik}\P^{jl} &=
    \P^k\paraa{\P^l(f)}-\P^l\paraa{\P^{k}(f)}
    +\D^i(f)\P^i\paraa{\P^{kl}}\\
    &= \P^k\paraa{\P^l(f)}-\P^l\paraa{\P^{k}(f)}+\{f,\P^{kl}\}\\
    &=\{x^k,\{x^l,f\}\}+\{x^l,\{f,x^k\}\}+\{f,\{x^k,x^l\}\}=0,
  \end{align*}
  again by using the Jacobi identity.
\end{proof}

\noindent It follows from the above result that
$[\D^i,\D^j](f)X^iY^j=0$ for all $X,Y\in\TSigma$ and $f\in
C^\infty(\Sigma)$, which implies that the Codazzi-Mainardi equations
in $(\reals^m,\delta)$:
\begin{align*}
  X^iY^jZ^k\paraa{\D^i\D^j(N^k)-\D^j\D^i(N^k)} = 0,
\end{align*}
are satisfied when assuming that \eqref{eq:PgDef} holds.


\bibliographystyle{alpha}
\bibliography{nambu_pseudo}

\end{document}